\documentclass[leqno]{amsart}
\usepackage{times}
\usepackage{amsfonts,amssymb,amsmath,amsgen,amsthm}
\usepackage{hyperref}
\usepackage{color}
\newcommand{\msc}[2][2000]{%
  \let\@oldtitle\@title%
  \gdef\@title{\@oldtitle\footnotetext{#1 \emph{Mathematics subject
        classification.} #2}}%
}

\theoremstyle{plain}
\newtheorem{theorem}{Theorem}[section]

\newtheorem{assumption}[theorem]{Assumption}
\newtheorem{lemma}[theorem]{Lemma}
\newtheorem{corollary}[theorem]{Corollary}
\newtheorem{proposition}[theorem]{Proposition}

\theoremstyle{remark}
\newtheorem{remark}[theorem]{Remark}

\def\C{{\mathbb C}}
\def\R{{\mathbb R}}
\def\N{{\mathbb N}}
\def\O{\mathcal O}

\def\({\left(}
\def\){\right)}
\def\<{\left\langle}
\def\>{\right\rangle}
\def\le{\leqslant}
\def\ge{\geqslant}

\def\d{{\partial}}

\def\l{\lambda}

\def\si{{\sigma}}

\DeclareMathOperator{\RE}{Re}
\DeclareMathOperator{\IM}{Im}

\numberwithin{equation}{section}

\begin{document}

\title[Sharp weights in  NLS with potential]{Sharp weights in the
  Cauchy problem for
  nonlinear Schr\"odinger equations with potential}

\author[R. Carles]{R\'emi Carles}
\address{CNRS \& Univ. Montpellier\\Math\'ematiques
\\CC~051\\34095 Montpellier\\ France}
\email{Remi.Carles@math.cnrs.fr}

\begin{abstract}
  We review different properties related to the Cauchy problem for the
  (nonlinear) Schr\"odinger equation with a smooth potential. For
  energy-subcritical nonlinearities and at most quadratic potentials,
  we investigate the necessary decay in space in order for the Cauchy problem
  to be locally (and globally) well-posed. The characterization of
  the minimal decay is
  different in the case of  super-quadratic potentials. 
\end{abstract}
\thanks{This work was supported by the French ANR projects
  SchEq (ANR-12-JS01-0005-01) and BECASIM
  (ANR-12-MONU-0007-04).} 
\maketitle

\section{Generalities}
\label{sec:gen}

We consider the nonlinear Schr\"odinger equation
\begin{equation}
  \label{eq:nls}
  i\d_t u = Hu +\lambda |u|^{2\si}u,\quad (t,x)\in \R\times \R^d,\quad
  H=-\frac{1}{2}\Delta +V(x), 
\end{equation}
with $\lambda \in \R$, $\si>0$, 
for some smooth, real-valued potential $V$. In order to work at the
level of regularity $L^2$ or $H^1$, we suppose that the nonlinearity
is energy-subcritical, that is, $\si<2/(d-2)$ when $d\ge 3$ (see
e.g. \cite{CazCourant}). Such models appear in various fields of
Physics, such as laser propagation or Bose-Einstein Condensation (see
e.g. \cite{PiSt,Sulem}): for instance the potential $V$ can be 
quadratic (harmonic potential), linear (Stark effect), or
super-quadratic to ensure a strong confinement. 

\begin{assumption}\label{hyp:V}
  We suppose that $V$ is smooth and real-valued, $V\in
  C^\infty(\R^d;\R)$, and,
  \begin{itemize}
  \item Either $V$ is at most quadratic, $\d^\alpha V\in
    L^\infty(\R^d)$ for all $\alpha\in \N^d$ with $|\alpha|\ge 2$,
\item Or $V\ge 0$ is super-quadratic, in the sense that $V(x)\to
  \infty$ as $|x|\to \infty$ and there exists $m> 2$ such
  that
  \begin{equation*}
    |\d^\alpha V(x)|\le C_\alpha \<x\>^{m-|\alpha|},\quad \forall  \alpha \in \N^d.
  \end{equation*}
  \end{itemize}
\end{assumption}
Typically, the second case addresses potentials of the form
$V(x)=\<x\>^m$, $m>2$. Of course, the second case is formally
compatible with the first one, but should be thought of as rather
complementary. The borderline case corresponds to quadratic
potentials. The fact that quadratic potentials play a special role has
been known for many years: as established in \cite{Yajima96}, the
fundamental solution associated to the linear solution is smooth and
bounded except at the initial time if $V(x)=o(|x|^2)$ at infinity, while
at least if $d=1$, the fundamental solution associated to
super-quadratic potentials is nowhere $C^1$. In the limiting exactly quadratic case, the
fundamental solution has isolated singularities
(\cite{Zelditch83}). Finally, the linear Schr\"odinger flow is not
uniquely defined for non-positive super-quadratic potentials: if for
instance $d=1$ and $V(x)=-x^4$, then $H$ is not essentially
self-adjoint on $C_0^\infty(\R^d)$, due to infinite speed of propagation in
the classical trajectories (see e.g. \cite{Dunford}). From this point
of view, for smooth potentials, the assumption that there exist
$a,b>0$ such that $V(x)\ge 
-a|x|^2-b$ is sharp in order to ensure that $H$ is essentially
self-adjoint on $C_0^\infty(\R^d)$ (\cite{ReedSimon2}).
\smallbreak

In this note, we review known results concerning the Cauchy
problem for \eqref{eq:nls} with a level of regularity $H^1(\R^d)$. The
only new result is Theorem~\ref{theo:new}, which shows that the sharp
weights in space, at the level of regularity $H^1$, have different characterizations
for at most quadratic potentials and for super-quadratic
potentials. The sharpness of the required decay in space is presented
in Proposition~\ref{prop:lin}. 

\section{Strichartz estimates}
\label{sec:strichartz}

For at most quadratic potentials, a parametrix for $e^{-itH}$ has been
constructed in \cite{Fujiwara79} (see also \cite{Kit80}). We simply
emphasize that as a consequence, the propagator $e^{-itH}$, which is
unitary on $L^2(\R^d)$, satisfies the following  local dispersive
estimate: there exist $C,\delta>0$ such that
\begin{equation*}
  \|e^{-itH}\|_{L^1(\R^d)\to L^\infty(\R^d)}\le
  \frac{C}{|t|^{d/2}},\quad |t|\le \delta.
\end{equation*}
Recall that under the general Assumption~\ref{hyp:V}, such an
estimate is necessarily local in time, since typically in the case of
the harmonic potential, the flow $e^{-itH}$ is periodic in time. 
\smallbreak

The case of super-quadratic potentials has been addressed in
\cite{YajZha01,YajZha04}, and in \cite{Miz14} (see also
\cite{AnFa09,RoZu08,Tho10}). We summarize the main 
results on Strichartz estimates for $e^{-itH}$ in the following
statement.
\begin{proposition}[From \cite{Fujiwara79,Miz14}]\label{prop:strichartz}
Let $d\ge 1$ and $V$ satisfying Assumption~\ref{hyp:V}. Let $(q,r)$ be
an admissible pair, that is, satisfying
\begin{equation*}
  \frac{2}{q}= d\(\frac{1}{2}-\frac{1}{r}\),
\end{equation*}
with $2\le q,r\le \infty$ and $q>2$ is $d=2$. Let $T>0$. 
\begin{itemize}
\item If $V$ is at most quadratic, there exists $C=C(q,d,T)$ such that
  \begin{equation*}
    \|e^{-itH}f\|_{L^q([-T,T];L^r(\R^d))}\le C
    \|f\|_{L^2(\R^d)},\quad \forall f\in L^2(\R^d).
  \end{equation*}
\item If $d\ge 3$ and $V$ is  super-quadratic, there exists $C=C(q,d,T)$ such that
  \begin{equation*}
    \|e^{-itH}f\|_{L^q([-T,T];L^r(\R^d))}\le C
    \|f\|_{B^{\frac{1}{q}\(1-\frac{2}{m}\)}},\quad \forall f\in B^{\frac{1}{q}\(1-\frac{2}{m}\)},
  \end{equation*}
where for $s\ge 0$, 
\begin{equation*}
  B^s=\left\{ f\in L^2(\R^d)\ ;\
      \|f\|_{B^s}:= \|H^{s/2}f\|_{L^2(\R^d)}<\infty\right\}.
\end{equation*}
\end{itemize}
\end{proposition}
We refer to \cite{YajZha01} and \cite{Miz14} for local in time
Strichartz estimates with super-quadratic potentials in the case $d=1$
and $d=2$, respectively. We omit them for the sake of
concision. Similar estimates are available for retarded
terms, which appear in the Duhamel's formula associated to
inhomogeneous Schr\"odinger equations. In the same fashion as above, 
these estimates are necessarily local in time without extra
assumption on $V$, since $H$ may possess eigenvalues (this is 
the case if $V\to +\infty$ as $|x|\to \infty$). 

The following equivalence of norms is established, in \cite{YajZha04}
when $V$ grows like $\<x\>^m$ as $|x|\to \infty$, and in
 \cite{BCM08}  under  assumptions
on $V$ which are weaker than in Assumption~\ref{hyp:V},
\begin{equation}\label{eq:equiv}
  \|H^{s/2}f\|_{L^2(\R^d)} \approx \|f\|_{H^s(\R^d)} +
  \|V^{s/2}f\|_{L^2(\R^d)}, \quad s\ge 0. 
\end{equation}

 As shown by the second case of Proposition~\ref{prop:strichartz}, a loss of regularity
must be expected for super-quadratic potentials.  reminiscent of what
happens on compact manifolds without boundary,
\cite{Bou93,BGT}. Formally, as $k$ ranges from $2$ to $\infty$, the
loss of regularity varies between $0$ and $1/q$ derivative, this
limiting case corresponding to the estimate established in \cite{BGT}
for general compact manifolds. This is in agreement with the property,
used in Physics, that for $V(x)=\<x\>^m$, the larger the
$m$, the more confining $H$. 
\section{Nonlinear Cauchy problem}
\label{sec:cauchy}

Formally, \eqref{eq:nls} enjoys the conservations of mass and energy:
\begin{equation*}
  \frac{d}{dt}\|u(t)\|_{L^2(\R^d)}^2=
\frac{d}{dt}\(\<Hu,u\>
+\frac{\l}{\si+1}\|u(t)\|_{L^{2\si+2}}^{2\si+2}\)=0. 
\end{equation*}

For $V\ge 0$ and at most quadratic, a local in time solution to
\eqref{eq:nls} was constructed by Oh \cite{Oh} with data in $\sqrt{
  H}$. We emphasize that the 
proof there does not rely on a fixed point argument, but on an
approximation procedure, as in \cite{GV79Cauchy}. In particular, it is
not necessary to understand the action of the
pseudo-differential operator $\sqrt H$ on the nonlinear term
$|u|^{2\si}u$. The case of the focusing,
$L^2$-subcritical nonlinearity considered in \cite{Oh} can easily be
generalized in view of the known results for $V=0$. 
\begin{theorem}[From \cite{Oh}]\label{theo:oh}
  Let $V\ge 0$ be at most quadratic, and $u_0\in \sqrt{H}=B^1$. 
  \begin{itemize}
  \item There exists a unique solution $u\in C([-T,T];\sqrt H)\cap
    L^{\frac{4\si+4}{d\si}}([-T,T];L^{2\si+2}(\R^d))$ to 
    \eqref{eq:nls}, with initial datum $u_0$, for some $T>0$ depending
    on $\|u_0\|_{B^1}$. 
\item This solution is global in time, $u\in C(\R;\sqrt H)\cap
  L^{\frac{4\si+4}{d\si}}_{\rm loc}(\R;L^{2\si+2}(\R^d))$, in either
  of the following cases:
  \begin{itemize}
  \item $\si<2/d$,
\item $\si\ge 2/d$ and $\l\ge 0$. 
  \end{itemize}
  \end{itemize}
\end{theorem}
The assumption on the sign of $V$ has been removed for initial
data in 
\begin{equation*}
  \Sigma = \left\{ f\in H^1(\R^d),\ \|f\|_\Sigma:=\|f\|_{H^1(\R^d)} +
    \|x f\|_{L^2(\R^d)}<\infty\right\}.
\end{equation*}
In \cite{Ca11}, the global Cauchy problem for \eqref{eq:nls} was
considered:
\begin{theorem}[From \cite{Ca11}]\label{theo:ca11}
  Let $V$ be at most quadratic, and $u_0\in \Sigma$. 
  \begin{itemize}
  \item There exists a unique solution $u\in C([-T,T];\Sigma
    )$ to
    \eqref{eq:nls}, with initial datum $u_0$, such that
    \begin{equation*}
      u,xu,\nabla u \in L^{\frac{4\si+4}{d\si}}([-T,T];L^{2\si+2}(\R^d)),
    \end{equation*}
for some $T>0$ depending
    on $\|u_0\|_{\Sigma}$. 
\item This solution is global in time, $u\in C(\R;\Sigma)$, $
      u,xu,\nabla u \in L^{\frac{4\si+4}{d\si}}_{\rm loc}(\R;L^{2\si+2}(\R^d)),$
in either
  of the following cases:
  \begin{itemize}
  \item $\si<2/d$,
\item $\si\ge 2/d$ and $\l\ge 0$. 
  \end{itemize}
  \end{itemize}
\end{theorem}
For $V\ge 0$ at most quadratic, \eqref{eq:equiv} shows that
$\Sigma\subset \sqrt H$,  and in the case where $V\ge
0$ is a non-degenerate quadratic form (e.g. isotropic harmonic
potential), $\Sigma = \sqrt H$. Therefore, the above result removes the sign
assumption in Theorem~\ref{theo:oh}, up to possibly requiring stronger decay in
space on the initial datum $u_0$. Note also that in the framework of
Theorem~\ref{theo:ca11}, even if $\lambda \ge 0$, the energy
functional 
\begin{align*}
  E &= \<Hu,u\>
+\frac{\l}{\si+1}\|u(t)\|_{L^{2\si+2}}^{2\si+2} \\
&= \frac{1}{2}\|\nabla
u(t)\|_{L^2}^2 +\int_{\R^d}V(x)|u(t,x)|^2dx
+\frac{\l}{\si+1}\|u(t)\|_{L^{2\si+2}}^{2\si+2} , 
\end{align*}
is not necessarily positive: it turns out that a negative (at most
quadratic) potential is not an 
obstruction to the existence of a solution to \eqref{eq:nls}. It may
actually prevent finite time blow-up (see e.g. \cite{Ca11} and
references therein). 
\smallbreak

We now turn to the case of a super-quadratic potential. 
It follows from the analysis in \cite{BCM08} that $B^s$ is a Banach
algebra for $s>d/2$ (see \eqref{eq:equiv}), and the following result
is proved:
\begin{theorem}[From \cite{BCM08}]\label{theo:BCM}
  Let $V$ be super-quadratic, $s>d/2$ and $u_0\in B^s$. Then for any
  $\si>0$ and $\l\in \C$, 
  there exist $T>0$ and a unique solution
  $u\in C([-T,T];B^s)$ to \eqref{eq:nls} with  initial datum $u_0$. 
\end{theorem}
To decrease the regularity, Strichartz inequalities make it possible
to prove:
\begin{theorem}[From \cite{YajZha04} and \cite{Miz14}] \label{theo:yajima}
  Let  $V$ be super-quadratic and $s\ge 0$ with
  \begin{equation*}
    s>\frac{d}{2} -\frac{1}{\si}\(\frac{1}{2}+\frac{1}{m}\). 
  \end{equation*}
Let $u_0\in B^s$. There exist $T>0$ and a unique solution
$C([-T,T];B^s)$ to \eqref{eq:nls} with initial datum $u_0$. (Uniqueness is
actually granted in smaller spaces, involving a mixed time-space norm
which we omit 
  to simplify the presentation.) 
\end{theorem}
Here again, for $m=2$, the above condition on $s$ is the standard one in the
case without potential (\cite{CW90}), and letting $m\to \infty$, we
recover the condition established in \cite{BGT} on compact manifolds
without boundary. To work at the level of $H^1$-regularity ($s=1$),
the above condition reads
\begin{equation*}
  \si<\frac{m+2}{m(d-2)_+}. 
\end{equation*}
In
view of the conservation of mass and energy, Gagliardo--Nirenberg
inequality, Theorem~\ref{theo:BCM} and 
Theorem~\ref{theo:yajima}  ($d=1$), or Theorem~\ref{theo:yajima}  and
\cite{BGT} ($d=2,3$) imply:
\begin{corollary}
  Let $d\le 3$ and $V$ be super-quadratic.
If  $s\ge 1$ and  $u_0\in B^s$, then \eqref{eq:nls} has a unique,
global solution $u\in C(\R;B^s)$ with initial datum $u_0$, in either
of the following cases:
 \begin{itemize}
  \item $\si<2/d$ and $\l\in \R$,
\item $2/d\le \si< (m+2)/(m(d-2)_+)$ and $\l\ge 0$. 
  \end{itemize}
\end{corollary}
\begin{remark}[Higher dimensions] 
  A similar result is available in higher dimensions, but the
  discussion is a bit more involved, since for $d\ge 4$,  it may happen
  that 
  \begin{equation*}
    \frac{2}{d}>\frac{m+2}{m(d-2)}. 
  \end{equation*}
\end{remark}
\section{Sharp weight for at most quadratic potentials}
\label{sec:sharp}

As noticed in \cite{CaCauchy}, if $V$ is at most quadratic and $\nabla
V$ is bounded, then Theorem~\ref{theo:ca11} remains valid with
$\Sigma$ replaced by $H^1(\R^d)$ (and no property involving
$xu$). Note that typically if $V(x)=\<x\>$, this shows that the
assumptions on the space decay of the initial datum are sharp neither
in Theorem~\ref{theo:oh} nor in Theorem~\ref{theo:ca11}, when one
wants to deal with an $H^1$-regularity. More generally, set
\begin{equation*}
  \widetilde\Sigma =\{f\in H^1(\R^d)\ ;\  \|f\|_{\widetilde \Sigma}
  :=\|f\|_{H^1(\R^d)}+\|f\nabla V\|_{L^2(\R^d)}<\infty\}. 
\end{equation*}
Since $V$ is at most quadratic, we have $\Sigma\subset \widetilde
\Sigma$, and the inclusion is strict unless $V$ is quadratric (and
non-degenerate). Typically, when $\nabla V\in L^\infty$, we have
$\widetilde\Sigma =H^1(\R^d)$. When $V\ge 0$, we also have $\sqrt
H \subset \widetilde \Sigma$, from the following elementary result.
\begin{lemma}
  Let $f\in C^2(\R;\R)$ be such that $f\ge 0$ and $f''$ is bounded. Then 
  \begin{equation*}
    f'(x)^2\le  2\|f''\|_{L^\infty} f(x),\quad \forall x\in \R. 
  \end{equation*}
\end{lemma}
\begin{proof}
 Taylor's formula yields, for $x,y\in \R$,
 \begin{equation*}
   f(x+y) = f(x)+yf'(x) + y^2\int_0^1 (1-\theta)f''(x+\theta
   y)d\theta\le  f(x)+yf'(x) + \frac{y^2}{2}\|f''\|_{L^\infty} .
 \end{equation*}
Since by assumption $f(x+y)\ge 0$, the discriminant of
$f(x)+yf'(x) + \frac{y^2}{2}\|f''\|_{L^\infty}$, seen as a polynomial
in $y$, 
 must be non-positive, hence the result. 
\end{proof}
\begin{theorem}\label{theo:new}
  Let $V$ be at most quadratic, and $u_0\in \widetilde\Sigma$. 
  \begin{itemize}
  \item There exists a unique solution $u\in C([-T,T];\widetilde\Sigma
    )$ to
    \eqref{eq:nls}, with initial datum $u_0$, such that
    \begin{equation*}
      u,u\nabla V,\nabla u \in L^{\frac{4\si+4}{d\si}}([-T,T];L^{2\si+2}(\R^d)),
    \end{equation*}
for some $T>0$ depending
    on $\|u_0\|_{\widetilde\Sigma}$. 
\item This solution is global in time, $u\in C(\R;\widetilde\Sigma)$, $
      u,u\nabla V,\nabla u \in L^{\frac{4\si+4}{d\si}}_{\rm loc}(\R;L^{2\si+2}(\R^d)),$
in either
  of the following cases:
  \begin{itemize}
  \item $\si<2/d$,
\item $\si\ge 2/d$ and $\l\ge 0$. 
  \end{itemize}
  \end{itemize}
\end{theorem}
\begin{proof}
  We sketch the main steps of the proof, which follow classical
  arguments. Duhamel's formula for \eqref{eq:nls}
with initial datum $u_0$ reads
\begin{equation*}
  u(t) = e^{-itH}u_0 -i\l\int_0^t
  e^{-i(t-s)H}\(|u|^{2\si}u\)(s)ds=:\Phi(u)(t). 
\end{equation*}
Local existence stems from a fixed point argument in a
  ball of the space
  \begin{equation*}
      X_T = \left\{u\in C([-T,T];\widetilde\Sigma)\ ;\ u,u\nabla V,\nabla u \in
    L^{\frac{4\si+4}{d\si}}\([-T,T];L^{2\si+2}(\R^d)\) \right\},
  \end{equation*}
for $T>0$ sufficiently small. Since local in time Strichartz estimates
are available for at most quadratic potentials, the only aspect which
differs from the usual approach where $V=0$ is that  $\nabla$ does not
commute with $H$, hence does not commute with $e^{-itH}$ for $t\not
=0$. However, we have the following commutator formulas,
\begin{equation*}
  \left[ i\d_t -H,\nabla\right]=\nabla V,\quad  \left[ i\d_t -H,\nabla
  V\right]=-\nabla^2 V\cdot \nabla-\frac{1}{2}\nabla \Delta V.
\end{equation*}
Since $\nabla^2 V$ and $\nabla \Delta V$ are bounded by assumption, we
get a closed system of 
estimates. In terms of $\Phi$, we have:
\begin{align*}
  \nabla \Phi(u)(t) &= e^{-itH}\nabla u_0 -i\l \int_0^t
 e^{-i(t-s)H} \nabla \(|u|^{2\si}u\)(s)ds\\
&\quad -i\int_0^t e^{-i(t-s)H}\(\Phi(u)(s)\nabla V\)ds,\\
\Phi(u)(t)\nabla V &=  e^{-itH}\( u_0\nabla V\) -i\l \int_0^t
 e^{-i(t-s)H} \(\(|u|^{2\si}u\)(s)\nabla V\)ds\\
+i\int_0^t & e^{-i(t-s)H}\(\nabla^2V\cdot \nabla\Phi(u)(s)\)ds
+\frac{i}{2}\int_0^t e^{-i(t-s)H}\(\Phi(u)(s)\nabla\Delta V\)ds.
\end{align*}
We refer to \cite{Ca11,CazCourant} for details on the fixed point
argument. 
\smallbreak

Showing that the solution is global if the nonlinearity is not both
focusing and mass critical or super-critical, follows from the following remark,
which stems from the formal conservation of the energy, and justified
by a regularizing argument: 
\begin{align*}
  \frac{d}{dt}\(\frac{1}{2}\|\nabla
u(t)\|_{L^2}^2 
+\frac{\l}{\si+1}\|u(t)\|_{L^{2\si+2}}^{2\si+2}\)&=
-\frac{d}{dt}\int_{\R^d}V(x)|u(t,x)|^2dx \\
& = -2\RE \int_{\R^d}V(x)\bar u(t,x)\d_t u(t,x)dx\\
& = \IM  \int_{\R^d}V(x)\bar u(t,x)\Delta u(t,x)dx\\
& =-\IM \int_{\R^d}\bar u(t,x)\nabla V(x)\cdot\nabla u(t,x)dx.
\end{align*}
On the other hand, we compute
\begin{align*}
   \frac{d}{dt}\int_{\R^d}|\nabla V(x)|^2|u(t,x)|^2dx& = -2\RE
   \int_{\R^d}|\nabla V(x)|^2\bar u(t,x)\d_t u(t,x)dx\\ 
& = \IM  \int_{\R^d}|\nabla V(x)|^2\bar u(t,x)\Delta u(t,x)dx\\
& =-2\IM \int_{\R^d}\bar u(t,x)\nabla^2V(x)\nabla V(x)\cdot\nabla u(t,x)dx.
\end{align*}
Let
\begin{equation*}
 \mathcal E_\l(t)= \frac{1}{2}\|\nabla
u(t)\|_{L^2}^2 
+\frac{\l}{\si+1}\|u(t)\|_{L^{2\si+2}}^{2\si+2} +  \int_{\R^d}|\nabla
V(x)|^2|u(t,x)|^2dx. 
\end{equation*}
In view of the above computations, Cauchy-Schwarz and Young inequalities yield
\begin{equation*}
  \frac{d\mathcal E_\l}{dt}\le \(1+2\|\nabla^2V\|_{L^\infty}\)\|u(t)
  \nabla V\|_{L^2}\|\nabla u(t)\|_{L^2} \lesssim \mathcal E_0(t).
\end{equation*}
Global existence readily follows from the local
theory when $\l\ge 0$. In the mass sub-critical focusing case, one can
invoke either Gagliardo-Nirenberg inequality, or global existence at the
$L^2$-level (see e.g. \cite{CazCourant}).  
\end{proof}
To conclude, we outline that if $\nabla V$ is unbounded, then for
\eqref{eq:nls} to possess an $H^1$ local solution, one has to assume
that $u_0\in \widetilde \Sigma$. This phenomenon is geometrical, in
the sense that it is present in the linear case $\l=0$. It remains in
the nonlinear setting, in the same spirit as in \cite{CaCauchy}. We
shall therefore present the linear result, and refer to
\cite{CaCauchy} for the adaptation to the nonlinear framework of \eqref{eq:nls}. 
\begin{proposition}\label{prop:lin}
  Let $V$ be at most quadratic, with $\nabla V\not\in
  L^\infty(\R^d)$. If $u_0\in H^2(\R^d)\setminus \widetilde \Sigma$,
  then for arbitrarily small time $\tau>0$, the
  solution $u\in C([0,\tau];L^2(\R^d))$ to
  \begin{equation}
    \label{eq:lin}
    i\d_t u = Hu\quad ;\quad u_{\mid t=0}=u_0,
  \end{equation}
satisfies $\nabla u(\tau,\cdot)\not \in L^2(\R^d)$. 
\end{proposition}
\begin{proof}
 The phenomenon related to Proposition~\ref{prop:lin} is a rotation in
 phase space: the presence of $V$ with an unbounded gradient causes
 the appearance of oscillations. To filter out these oscillations,
 introduce the eikonal equation 
\begin{equation}
  \label{eq:eikonal}
  \d_t \phi +\frac{1}{2}|\nabla \phi|^2+V=0\quad ;\quad \phi_{\mid
    t=0}=0. 
\end{equation}
It is classically solved locally in time thanks to the Hamiltonian
flow
\begin{equation*}
   \dot x(t,y)=\xi(t,y),\quad \dot \xi(t,y) = -\nabla V\(x(t,y)\),\quad
  x(0,y)=y,\quad \xi(0,y)=0. 
\end{equation*}
For $V$ at most quadratic, there exists $T>0$ such that the map $y\mapsto x(t,y)$ is
invertible for $(t,y)\in [0,T]\times\R^d$. Such a time $T$ must be
expected to be necessarily 
finite in general, due to the formation of caustics.
Then \eqref{eq:eikonal} has a unique local smooth solution $\phi\in
C^\infty([0,T]\times \R^d)$, which is at most quadratic in space,
\begin{equation*}
  \d_x^\alpha \phi \in L^\infty([0,T]\times \R^d),\quad \forall \alpha
  \in \N^d,\ |\alpha|\ge 2. 
\end{equation*}
Details can be found for instance in \cite{CaBKW}. 
Introduce $a$ given by  $u(t,x)=
a(t,x)e^{i\phi(t,x)}$. On $[0,T]\times \R^d$, \eqref{eq:lin} is
equivalent to
\begin{equation}
\label{eq:transport}
\d_t a +\nabla \phi\cdot \nabla a +\frac{1}{2}a\Delta \phi
=\frac{i}{2}\Delta a\quad ;\quad a_{\mid t=0}=u_0.   
\end{equation}
This equation is a transport equation (left hand side), plus a
skew-symmetric term (right hand side). 
Since $\phi$ is at most quadratic in space and $u_0\in H^2(\R^d)$, we
check that \eqref{eq:transport} has a
unique solution $a\in C([0,T];H^2(\R^d))$. On the other hand, since $\phi\in
C^\infty([0,T]\times \R^d)$, \eqref{eq:eikonal} implies
\begin{equation*}
 |\d_t \nabla \phi+\nabla V|=|\nabla^2 \phi \cdot \nabla \phi|\lesssim
 |\nabla \phi| \lesssim |\nabla \phi +t\nabla V|+|t\nabla V|.
\end{equation*}
Gronwall lemma yields
\begin{equation*}
  |\nabla \phi(t,x) +t\nabla V(x)|\lesssim |\nabla V(x)|\int_0^t
  se^{Cs}ds\le C_T t^2  |\nabla V(x)|,\quad t\in [0,T].
\end{equation*}
To conclude, we approximate $a$ for short time. Introduce $\tilde a$
solution to 
\begin{equation}
\label{eq:transport2}
\d_t \tilde a +\nabla \phi\cdot \nabla \tilde a +\frac{1}{2}\tilde a\Delta \phi
=0\quad ;\quad \tilde a_{\mid t=0}=u_0.   
\end{equation}
It is given explicitly by the expression (see e.g. \cite{CaBKW})
\begin{equation*}
  \tilde a (t,x) = \frac{1}{\sqrt{ J_t\(y(t,x)\)}}u_0\(y(t,x)\), 
\end{equation*}
where $y(t,x)$ is the inverse map of $y\mapsto x(t,y)$ (well-defined
on $[0,T]\times \R^d$), and $J_t$ is the Jacobi determinant
\begin{equation*}
  J_t(y) =\operatorname{det}\nabla_y x(t,y), 
\end{equation*}
which is bounded away from zero and infinity on  $[0,T]\times
\R^d$. Subtracting \eqref{eq:transport2} from \eqref{eq:transport},
multiplying by the conjugate of $a-\tilde a$ and integrating by parts, we
get
\begin{equation*}
  \frac{d}{dt}\|a(t)-\tilde a(t)\|_{L^2}^2= \IM
  \int_{\R^d}(a(t,x)-\tilde a(t,x))\Delta \bar a(t,x)dx.
\end{equation*}
Cauchy-Schwarz inequality and Gronwall lemma yield (recall that $a\in
C([0,T];H^2)$ since we have assumed $u_0\in H^2(\R^d)$)
\begin{equation*}
  \|a(t)-\tilde a(t)\|_{L^2} \le Ct \|a\|_{L^\infty([0,T];H^2(\R^d))}.
\end{equation*}
Now
\begin{equation*}
  \nabla u(t,x) = e^{i\phi(t,x)}\nabla a (t,x) + e^{i\phi(t,x)}
  a(t,x)\nabla \phi(t,x),
\end{equation*}
with  $\nabla
\phi(t,x)=t\(1+\O(t)\)\nabla V(x)$ pointwise, and $a =\tilde a+\O(t)$
in $L^2(\R^d)$, as $t\to 0$.  Therefore, for arbitrarily small
time $\tau>0$, $\nabla u(\tau,\cdot)\not \in L^2(\R^d)$, hence the
result. 
\end{proof}
\begin{remark}
In the case  $V(x)=\<x\>^k$, $k>2$, \eqref{eq:equiv} shows that
$ \widetilde \Sigma\subsetneq B^1=\sqrt H\subsetneq \Sigma$. In view 
of Proposition~\ref{prop:strichartz}, and since the proof of
Proposition~\ref{prop:lin} heavily relies on the fact that $V$ is at
most quadratic, this suggests that 
for super-quadratic potentials, the weakest possible weight in space
at the $H^1$ level of regularity
corresponds to $\sqrt H$, that is, the property $u_0\sqrt V \in L^2(\R^d)$ (as in
Theorems~\ref{theo:BCM} and \ref{theo:yajima}).    
\end{remark}

\subsubsection*{Acknowledgements}
The author is grateful to Prof. Claude Zuily for several remarks on
this paper. 
\bibliographystyle{siam}

\bibliography{../../carles}

\end{document}